\documentclass[11pt]{amsart}
\usepackage[marginratio=1:1,height=600pt,width=400pt,tmargin=117pt]{geometry}
\usepackage{amsmath,amsthm,amssymb,amsbsy}
\usepackage{comment}
\usepackage{graphicx}
\usepackage{enumitem}
\usepackage{cleveref}
\usepackage{mathtools}

\newtheorem{theorem}{Theorem}[section]
\newtheorem{lemma}[theorem]{Lemma}
\crefname{lemma}{Lemma}{Lemmas}
\newtheorem{proposition}[theorem]{Proposition}
\newtheorem{corollary}[theorem]{Corollary}

\theoremstyle{definition}
\newtheorem{definition}[theorem]{Definition}

\theoremstyle{remark}
\newtheorem{remark}[theorem]{Remark}
\newtheorem{claim}{Claim}

\numberwithin{equation}{section}

\newcommand{\R}{\mathbb{R}}
\newcommand{\N}{\mathbb{N}}
\newcommand{\Z}{\mathbb{Z}}

\DeclareMathOperator{\divr}{div}

\DeclareMathOperator{\supp}{supp}

\DeclareMathOperator{\diam}{diam}

\def\logoversion{squarelogo}
\relax

\RequirePackage{graphicx} 
\RequirePackage{ifthen}   

\def\degreedate#1{\gdef\@degreedate{#1}}
\def\degree#1{\gdef\@degree{#1}}
\def\college#1{\gdef\@college{#1}}

\ifthenelse{\equal{\logoversion}{shieldcrest}}%
{

}{}

\ifthenelse{\equal{\logoversion}{beltcrest}}%
{

}{}

\ifthenelse{\equal{\logoversion}{squarelogo}}%
{

}{}

\title[the Navier-Stokes equations]
      {Partially regular weak solutions of the stationary Navier-Stokes equations in dimension 6}

\author[Bian Wu]{Bian Wu}
\email{bian.wu@math.ethz.ch}
\date{\today}

\begin{document}

\begin{abstract}
By using defect measures, we prove the existence of partially regular weak solutions to the stationary Navier-Stokes equations with external force $f \in L_{\text{loc}}^q \cap L^{3/2}, q>3$ in general open subdomains of $\R^6$. These weak solutions satisfy certain local energy estimates and we estimate the size of their singular sets in terms of Hausdorff measures. We also prove the defect measures vanish under a smallness condition, in contrast to the nonstationary Navier-Stokes equations in $\R^4 \times [0,\infty[$.
\end{abstract}

\maketitle

\section{Introduction}

\subsection{Main result}
This is a companion article of \cite{wu2021partially} where we prove the existence of a partially regular weak solutions to the nonstationary Navier-Stokes equations in $\R^4 \times [0,\infty[$, i.e.
\begin{equation} \label{navierstokes}
  \begin{split}
  \partial_t u - \Delta u + (u \cdot \nabla) u + \nabla p &= f, \\
  \divr u &= 0.
  \end{split}
\end{equation}
\par

In space dimensions $n \geq 2$, the existence of weak solutions to \eqref{navierstokes} in the whole space and in arbitrary open domains was proved by Leray \cite{leray1934mouvement} and Hopf \cite{hopf1950anfangswertaufgabe}, respectively in 1934 and 1950. After Scheffer' works \cite{scheffer1977hausdorff,scheffer1978navier,scheffer1980navier} on the partial regularity theory, Caffarelli, Kohn and Nirenberg \cite{caffarelli1982partial} proved the existence of suitable weak solutions and the fact that $1$-dimensional parabolic Hausdorff measure of the singular set of these solutions is zero in $\R^3 \times [0,\infty[$. The suitable weak solutions satisfy a local energy inequality which allows local analysis. In $\R^4 \times [0,\infty[$, a compactness property which is crucial for proving the existence of suitable weak solutions is missing. In \cite{wu2021partially}, we overcome it by introducing defect measures and incorporating it into energy estimates.\par

For the stationary Navier-Stokes equations
\begin{equation} \label{sta_navierstokes}
  \begin{split}
  - \Delta u + (u \cdot \nabla) u + \nabla p &= f,\\
  \divr u &= 0,
  \end{split}
\end{equation} 
Struwe \cite{struwe1988partial} proved partial regularity result in $\R^5$ in 1982. Later, assuming the existence of suitable weak solutions, Dong and Du \cite{dong2007partial} proved similar results for the nonstationary equations in $\R^4 \times [0,\infty[$, and Dong and Strain \cite{dong2012partial} adapted this method to the stationary equations in $\R^6$. In this article, we prove the existence of weak solutions satisfying local energy estimates in $\R^6$, under very general assumptions on the external force $f$. We also prove these solutions are partially regular. The main theorem is the following.
\begin{theorem} \label{maintheorem}
Given a weakly solenoidal force $f \in L_{\text{loc}}^q(\R^6) \cap L^{3/2}(\R^6), q>3$, there exists a weak solution set $(u,p,\nu,\mu)$ for the stationary Navier-Stokes equations \eqref{sta_navierstokes} in $\R^6$ which satisfies local energy inequalities \eqref{local_energy_1} and \eqref{local_energy_2}. Here, $(u,p)$ is a weak solution of the stationary Navier-Stokes equations with $u \in \dot{H}(\R^6)$ and $p \in L^{3/2}(\R^6)$. Moreover, the singular set $S$ of $u$ satifies $\mathcal{H}^2(S) < C$ for some constant $C$ depending on $f$.
\end{theorem}
The definitions of weak solution sets and the singular sets are given in \Cref{weaksolutionset} and \Cref{singularset}. Although we state \Cref{maintheorem} for $\R^6$, the nature of our analysis is local and it has almost no requirement on the domain. \par

If the spatial domain and the force $f$ are sufficiently regular, the existence of a regular solution is known in many cases, while its uniqueness is open. In the dimensions $2$, $3$ and $4$, we refer to Galdi \cite{galdi2011introduction} for the existence of regular solutions. In dimension $5$, the existence of regular solutions was independently proved by Struwe \cite{struwe1995regular} in  $\R^5$ and by Frehse, R{\u{u}}{\v{z}}i{\v{c}}ka in \cite{frehse1995existence} in $\R^5/\Z^5$. For $f \in L^{\infty}$ and bounded subdomains of $\R^6$, Frehse and R{\u{u}}{\v{z}}i{\v{c}}ka proved the existence of regular solutions in \cite{frehse1994regularity1} and \cite{frehse1996existence}. They showed there exist weak solutions $(u,p)$ such that $({|u|}^2+2p)_+$ is bounded in any compact subdomain, based on their earlier work \cite{frehse1994regularity}. Struwe's method or Frehse, R{\u{u}}{\v{z}}i{\v{c}}ka's method use global peoperties such as maximum principle or Green function of the domain, thus it relies on the specific structures of the equations in certain domains and the force $f$. In contrast, our method is completely different and based on local analysis, and can be applied to more general problem like nonstationary Navier-Stokes equations \cite{wu2021partially} and incompressible magneto-hydrodynamic equations. \par

As we mentioned in \cite{wu2021partially}, similar to the nonstationary Navier-Stokes equations in $\R^4 \times [0,\infty[$, the stationary equations in $\R^6$ is critcial since approximation solutions are not relatively compact in $L^3$, so it is subtle to deduce local energy inequality for the limiting solution. We overcome this by proving slightly weaker local energy inequalities containing concentration measures. These concentration measures account for the loss of compactness in $L^3$ and appear naturally in the construction of our weak solutions. For this purpose, we introduce a new notion of generalized solutions to incorporate these measures, namely weak solutions sets defined in \Cref{weaksolutionset}. These local energy inequalities are sufficient to give all nice partial regularity results analogue to those obtained by Caffarelli, Kohn and Nirenberg in their celebrated paper \cite{caffarelli1982partial} or those in our previous work \cite{wu2021partially}. \par

\subsection{Comparison with nonstationary setting}
On the other hand, we also find the stationary equations in $\R^6$ is different from the nonstationary equations in $\R^4 \times [0,\infty[$ when studying the concentration phenomena in approximation procedure. The first main difference shows up when proving the compactness of the concentration measures in \Cref{tightness}. In nonstationary setting \cite{wu2021partially}, we use time variable to control the vanishing at infinity. Here, we need to localize the solution to exclude the mass concentrating at infinity.

The second difference is the vanishing of the concentration measures. Although one dimension in time counts for two dimensions in space, the weak solutions are less regular in time than in space. As a result, the concentration measures vanish under a local smallness condition in the stationary setting, i.e.
\begin{corollary} \label{ckn1_scaled}
Given a weakly solenoidal force $f \in L_{\text{loc}}^q(\R^6) \cap L^{3/2}(\R^6), q>3$, there exists absolute constants $\epsilon>0, \kappa>0$, and for any fixed $q > 3$, there exists another constant $C=C(q)$ satisfying the following property: if a weak solution set $(u,p,\nu,\mu)$ of the Navier-Stokes equations in $\R^6$ with a weakly solenoidal force $f \in L_{\text{loc}}^q(\R^6), q>3$ satisfies
\begin{equation} \label{ckn1_scaled_initial}
  \begin{split}
  r^{-3} \int_{B_r(x_0)} (|u|^3 dx + d\nu) + r^{-3} \int_{B_r(x_0)} |p|^{3/2} dx &\leq \epsilon \\
  r^{3q-6}\int_{B_r(x_0)} |f|^q dx &\leq \kappa
  \end{split}
\end{equation}
then $\|u\|_{L^{\infty}(B_{r/2}(x_0))} < C $ and $\nu|_{B_r(x_0)}= \mu|_{B_r(x_0)} = 0$.
\end{corollary}
\noindent In the nonstationary case in $\R^4 \times [0,T]$, whether these concentration measures vanish locally or not is not known. This is because the concentration measures in stationary case have more structures. They exist in the form of countably many atoms. For the nonstationary equations in $\R^4 \times [0,\infty[$, these measures may concentrate in a set of finite $1$-dimensional Hausdorff measure, for example, a curve. \par

This rest of this paper consists of two sections. In \Cref{section_local_energy}, we prove the existence of weak solution sets and local energy estimates. In \Cref{section_partial_regularity} we use the local energy estimates to obtain partial regularity results for weak solution sets.

\subsection{Acknowledgements} I would like to thank my advisor Michael Struwe for reading my draft and giving me helpful feedback.


\section{Existence of weak solution sets} \label{section_local_energy}

To construct generalized solutions to the Navier-Stokes equations, we start with constructing solutions to the regularized Navier-Stokes equations.

\begin{lemma} \label{preparation_existence}
Let $\{\chi_k\}_{k \in \N}$ be a sequence of standard mollifiers and $f \in L^{3/2}(\R^6)$, then we have a sequence $\{(u_k,p_k)\}_{k \in \N} \subset \dot{H}^1(\R^6) \times L^{3/2}(\R^6)$ such that $(u_k,p_k)$ is a distributional solution to the regularized stationary Navier-Stokes equations
\begin{equation} \label{reg_sta_navierstokes}
  \begin{split}
  - \Delta u_k + [(\chi_k * u_k) \cdot \nabla] u_k + \nabla p_k &= f, \\
  \divr u_k &= 0.
  \end{split},
\end{equation}
Here, $\{u_k\}_{k \in \N}$ is uniformly bounded in $\dot{H}^1(\R^6)$ and $\{p_k\}_{k \in \N}$ is uniformly bounded in $L^{3/2}(\R^6)$, thereby we can pass to the weak limit, i.e.
\begin{equation} \label{preparation_existence_eq0}
  \begin{split}
  u_k \rightarrow u \quad &\text{weakly in} \quad \dot{H}^1(\R^6), \\
  p_k \rightarrow p \quad &\text{weakly in} \quad L^{3/2}(\R^6).
  \end{split}
\end{equation}
Moreover, this sequence satisfies local energy inequality, i.e. for any $\phi \in C_c^{\infty}$ there holds
\begin{equation} \label{local_energy_seq}
  \begin{split}
  2\int_{\R^6} |\nabla u_k|^2 \phi dx 
  \leq& \int_{\R^6} |u_k|^2 |\Delta \phi| dx + \int_{\R^6} |u_k|^2 (\tilde{u}_k \cdot \nabla) \phi dx \\
      &+ \int_{\R^6} 2 p_k (u_k \cdot \nabla) \phi dx + \int_{\R^6} f \cdot u_k \phi dx,
  \end{split}
\end{equation}
where $\tilde{u}_k := \chi_k * u_k$.
\end{lemma}

\begin{proof}
The existence of weak solutions of the regularized Navier-Stokes equations follows from Galerkin method. For a standard argument of Galerkin method, we refer to Theorem 4.4 and Theorem 14.1 in \cite{robinson2016three}. Galerkin method also yields that $u_k$ is bounded in $\dot{H}^1(\R^6)$, thereby the boundedness of $p_k$ follows from the Laplace equation for $p_k$ and Calderon-Zygmund theory. Standard regularity theory shows $u_k \in W^{2,3/2}_{\text{loc}}(\R^6)$, hence $u_k \phi$ is an admissible test function. This gives the local energy inequality \eqref{local_energy_seq}.
\end{proof}

\begin{remark} \label{remark_general_domains}
This result also holds for general open Lipschitz domains $\Omega \subset \R^6$ for prescribed boundary data $u_0 \in L^s(\partial \Omega), s \in [1,\frac{5}{2})$. Because of the compact trace operator $\dot{H}^{1}(\Omega) \rightarrow L^s(\partial \Omega)$ for $s \in [1,\frac{5}{2})$, the boundary value of the weak limit $u$ coincides with the fixed boundary value of the approximation sequence. The rest of our argument is purely local and it has no requirement on the domain.
\end{remark}

\subsection{Concentration compactness}
To study possible concentration phenomena appearing in the local energy inequality \eqref{local_energy_seq}, we need a variant of the concentration-compactness lemma by Lions in \cite{lions1985concentration}.

\begin{lemma} \label{variant_cc}
Suppose $\{\mu_m\}_{m \in \N}$ is a sequence of finite Borel measures on $\R^n$ with
\[ \lim_{m \rightarrow \infty} \mu_m(\R^n) = 1. \]
Then up to a subsequence, one of the following conditions holds:
\begin{enumerate}[leftmargin=*]
  \item (Compactness) For any $\epsilon > 0$ there is a radius $r_0>0$ such that
  \[ \int_{B_{r_0}} d \mu_m \geq \mu_m(\R^n) - \epsilon. \]
  \item (Vanishing) For all $r>0$ there holds
  \[ \lim_{m \rightarrow \infty} \int_{B_r} d\mu_m = 0. \]
  \item (Dichotomy) There exists $\lambda \in (0,1)$, such that for any $\epsilon>0$, there is a radius $r_0 > 0$ with following property. Given any $r'>r_0$, the measures $\mu_m^1 := \mu_m|_{B_{r_0}}$ and $\mu_m^2 := \mu_m|_{\R^n \backslash B_{r'}}$ satisfy
  \[
  \begin{split}
  &0 \leq \mu_m^1 + \mu_m^2 \leq \mu_m, \\
  &\limsup_{m \rightarrow \infty} \Big( \Big|\lambda - \int_{\R^n} d\mu_m^1\Big| 
                                      + \Big|\mu_m(\R^n)-\lambda - \int_{\R^n} d\mu_m^2\Big| \Big) \leq 2\epsilon.
  \end{split}
  \]
\end{enumerate}
\end{lemma}

\begin{remark}
In calculus of variations, Lions's concentration compactness lemma, i.e. Lemma I.1 in \cite{lions1985concentration}, studies the concentration points in a measure-theoretic level. Using this, one can show, for instance, that any minimizing sequence for the Sobolev constant in the embedding $\dot{H}(\R^6) \hookrightarrow L^{3}(\R^6)$ is compact up to translation and dilation. However, unlike searching minimizers for Sobolev constant, one cannot translate or dilate the solutions of regularized equations \eqref{reg_sta_navierstokes} with different scaling factors, because this also scales the force $f$. Instead, we show the convergence of certain measures by excluding the part which concentrates at infinity, and a simple observation shows that the remaining part also converges weakly in the sense of measures.
\end{remark}

\begin{proof}[Proof of \Cref{variant_cc}]
For a nonegative measure $\mu$, we define its concentration function
\[ Q(r) := \int_{B_r} d \mu. \]
Let $\{Q_m\}_{m \in \N}$ be the concentration functions of $\{\mu_m\}_{m \in \N}$, then 
\[ \lim_{m \rightarrow \infty} \lim_{r \rightarrow \infty} Q_m(r) = 1. \]
Up to a subsequence, this sequence converges to a non-decreasing, non-negative bounded function $Q$ at almost every point. And we can make $Q$ continuous from the left, then we have
\[ Q(r) \leq \liminf_{m \rightarrow \infty} Q_m(r). \]
\par

Let $\lambda = \lim_{r \rightarrow \infty} Q(r) \in [0,1]$. If $\lambda = 0$, this is the case of vanishing.
\par

If $\lambda = 1$, for any $\epsilon > 0$, there exists $r_0 > 0$ with $Q(r_0) \in \big(1-\frac{\epsilon}{2}, 1+\frac{\epsilon}{2}\big)$. For sufficiently large $m_0 \in \N$, we have $Q_m(r_0) > \mu_m(\R^n)-\epsilon$ for any $m \geq m_0$. Choose $r_0$ larger, if necessary, then we have
\[ Q_m(r_0) = \int_{B_{r_0}} d \mu_m > \mu_m(\R^n) - \epsilon \quad \text{for all } m \in \N. \]
\par

If $\lambda \in (0,1)$, for any $\epsilon > 0$, we can choose $r_0 > 0$ such that
\[  \lim_{m \rightarrow \infty} Q_m(r) = Q(r) 
    \in \Big(\lambda-\frac{\epsilon}{2}, \lambda\Big] 
    \quad \text{for } r=r_0 \text{ and almost every } r>r_0. \]
Given any $r'>r_0$, there exists $m_0(\epsilon,r') \in \N$ such that
\[ Q_m(r_0),Q_m(r') \in (\lambda-\epsilon, \lambda+\epsilon) \quad \text{for all } m \geq m_0. \]
Let $\mu_m^1 = \mu_m|_{B_{r_0}}$ and $\mu_m^2 = \mu_m|_{\R^n \backslash B_{r'}}$. This gives the case dichotomy.
\end{proof}
\par

Now, using \Cref{variant_cc}, we can prove the tightness of the measures showing up in \eqref{local_energy_seq}, after splitting the part which concentrate at infinity.

\begin{proposition} \label{tightness}
Given $f \in L^{3/2}(\R^6)$, let $u$ be as in \Cref{preparation_existence}. There exists $\{v_k\}_{k \in \N} \subset \dot{H}(\R^6)$ such that $v_k \rightarrow u$ weakly in $\dot{H}(\R^6)$ and the measures $\{|\nabla v_k|^2dx\}_{k \in \N}$, $\{|v_k|^3dx\}_{k \in \N}$ are tight, i.e. they satisfy the compactness condition in \Cref{variant_cc}. Moreover, for any $\phi \in C_c^{\infty}(\R^6)$ there exists $k_0 \in \N$ such that for $k \geq k_0$,
\begin{equation} \label{v_local_energy_inequality}
    \begin{split}
    2\int_{\R^6} |\nabla v_k|^2 \phi dx \leq& \int_{\R^6} |v_k|^2 |\Delta \phi| dx
      + \int_{\R^6} |v_k|^2 (\tilde{v}_k \cdot \nabla) \phi dx \\
      &+ \int_{\R^6} p_k (v_k \cdot \nabla) \phi dx + \int_{\R^6} f v_k \phi dx .
    \end{split}
\end{equation}
\end{proposition}

The proof of \Cref{tightness} relies on the following obsertvations.
\begin{lemma} \label{vanishing_weak_limit}
Suppose $\{h_k\}_{k \in \N} \subset \dot{H}^1(\R^n), n \geq 1$ is bounded and for any $\rho>0$ there holds
\[ \lim_{k \rightarrow \infty} \int_{B_\rho} |\nabla h_k|^2 dx = 0. \]
Then $h_k \rightarrow 0$ weakly in $\dot{H}^1(\R^n)$.
\end{lemma}

\begin{proof}
It suffices to show $(h_k,\varphi)_{\dot{H}^1} \rightarrow 0$ for any $\varphi \in C_c^{\infty}(\R^n)$. This is straightforward.
\end{proof}

\begin{lemma} \label{vanishing_weak_limit_L3}
Suppose $\{h_k\}_{k \in \N} \subset \dot{H}^1(\R^n), n \geq 3$ is bounded and for any $\rho>0$ there holds
\[ \lim_{k \rightarrow \infty} \int_{B_\rho} |h_k|^{2n/(n-2)} dx = 0. \]
Then $h_k \rightarrow 0$ weakly in $L^{2n/(n-2)}(\R^n)$ and $\dot{H}^1(\R^n)$ up to a subsequence.
\end{lemma}

\begin{proof}
For the weak convergence in $L^{2n/(n-2)}(\R^n)$, it suffices to show
\[ (h_k,\varphi)_{L^{2n/(n-2)}(\R^n) \times L^{2n/(n+2)}(\R^n)} \rightarrow 0 \]
for any $\varphi \in C_c^{\infty}(\R^n)$. This is straightforward by splitting the integral into $B_\rho$ and $\R^n \backslash B_\rho$. Since $\{h_k\}_{k \in \N}$ is bounded in $\dot{H}^1(\R^n)$, $h_k \rightarrow h_0$ weakly in $\dot{H}^1(\R^n)$ up to a subsequence. Hence $h_k \rightarrow h_0$ weakly in $L^{2n/(n-2)}(\R^n)$ up to a subsequence. And this gives $h_0 = 0$.
\end{proof}

\begin{proof}[Proof of \Cref{tightness}]
This idea of this proof is to apply \Cref{variant_cc} to $\{|\nabla u_k|^2dx\}_{k \in \N}$ and $\{|v_k|^3dx\}_{k \in \N}$ separately. The case of vanishing and the case of compactness are easy. In the case of dichotomy, we need to do suitable localization to split the concentration at infinity.
\par

First, we look at the measures $\{|\nabla u_k|^2dx\}_{k \in \N}$. There exists a subsequence $\{u_k\}_{k \in \Lambda}$ such that
\[ \liminf_{k \rightarrow \infty} \|\nabla u_k\|^2_{L^2} = \lim_{k \rightarrow \infty, k \in \Lambda} \|\nabla u_k\|^2_{L^2}, \]
then without loss of generality, we assume
\[ \liminf_{k \rightarrow \infty} \|\nabla u_k\|^2_{L^2} = \lim_{k \rightarrow \infty} \|\nabla u_k\|^2_{L^2} = 1. \]
\par

We apply \Cref{variant_cc} to the measures $\{|\nabla u_k|^2dx\}_{k \in \N}$. In the case of vanishing, \Cref{vanishing_weak_limit} shows $u_k \rightarrow 0$ weakly in $\dot{H}(\R^6)$, then $v_k = 0$ satisfies all the conditions.
\par
In the case of compactness, the set of measures $\{|\nabla u_k|^2dx\}_{k \in \N}$ is tight. $v_k = u_k$ satisfies all the conditions except for the tightness of $\{|v_k|^3dx\}_{k \in \N}$.
\par
In the case of dichotomy, we can choose $\epsilon_k \rightarrow 0$ with corresponding $r_k \rightarrow \infty$ such that
\begin{equation} \label{tightness_eq0}
  \limsup_{k \rightarrow \infty} \Big( \Big|\lambda - \int_{B_{r_k}} |\nabla u_k|^2 dx\Big| 
      + \Big|\mu_k(\R^6) -\lambda - \int_{\R^n \backslash B_{2r_k}} |\nabla u_k|^2 dx\Big| \Big) = 0.
\end{equation}
Define the cut-off functions $\{\varphi_k\}_{k \in \N}$ satisfying
\[ 0 \leq \varphi_k \leq 1, \quad \varphi_k|_{B_{r_k}} = 1, 
    \quad \varphi_k|_{\R^6 \backslash B_{2r_k}} = 0, \quad |\nabla \varphi_k| \leq Cr_k^{-1}, \quad |\nabla^2 \varphi_k| \leq Cr_k^{-2}. \]
Let $v_k := u_k \varphi_k$, $w_k := u_k(1-\varphi_k)$, then we can deduce that, for any $\rho>0$,
\[ \lim_{k \rightarrow \infty} \int_{B_\rho} |\nabla w_k|^2 dx = 0. \]
\Cref{vanishing_weak_limit} yields $w_k \rightarrow 0$ weakly in $\dot{H}(\R^6)$, and $v_k \rightarrow u$ weakly in $\dot{H}(\R^6)$. Now let $Q_k$ be the concentration functional of the measure $|\nabla v_k|^2dx$ as in \Cref{variant_cc}. Then \eqref{tightness_eq0} yields
\[
  \lim_{k \rightarrow \infty} Q_k(r_k) = \lambda,
  \quad
  \lim_{k \rightarrow \infty} \int_{\R^6} |\nabla v_k|^2 dx = \lambda.
\]
According to the proof of \Cref{variant_cc}, $Q_k$ converges at almost every point to a non-decreasing, non-negative, bounded and left continuous function $Q$. In the dichotomy case, for any $\epsilon > 0$, there exists $r_0>0$ such that $\lim_{k \rightarrow \infty} Q_k(r_0) = Q(r_0)$ and
\[ \limsup_{k \rightarrow \infty} \Big| \lambda - \int_{B_{r_0}} |\nabla u_k|^2 dx \Big| \leq 2 \epsilon. \]
For $k$ sufficiently large, $r_k \geq r_0$, then $u_k = v_k$ in $B_{r_0}$. This gives
\[ \limsup_{k \rightarrow \infty} \Big| \lambda - \int_{B_{r_0}} |\nabla v_k|^2 dx \Big| \leq 2 \epsilon. \]
By the convergence of $Q_k(r_0)$, we have $Q(r_0) \in (\lambda - 2 \epsilon,\lambda + 2 \epsilon)$. Therefore, $\lim_{r \rightarrow \infty} Q(r) = \lambda$. Using the same argument for the compact case in the proof of \Cref{variant_cc}, we know $\{|\nabla v_k|^2dx\}_{k \in \N}$ is tight.
\par

For the compactness of the measures $\{|v_k|^3dx\}_{k \in \N}$, we assume, without loss of generality,
\[ \liminf_{k \rightarrow \infty} \|v_k\|^3_{L^3} = \lim_{k \rightarrow \infty} \|v_k\|^3_{L^3} = 1 \]
\par
Apply \Cref{variant_cc}. In the case of vanishing, \Cref{vanishing_weak_limit_L3} shows that $v_k \rightarrow 0$ weakly in $\dot{H}(\R^6)$ up to a subsequence, then $v_k = 0$ satisfies all the conditions. In the case of compactness, the set of measures $\{|v_k|^3dx\}_{k \in \N}$ is tight.
\par

In the case of dichotomy, we get a sequence of enlarged radius $r_k \rightarrow +\infty$ and we can define cutoff functions as before. With the same argument, the products of $v_k$ and these cutoff functions are the final approximation candidates. With slight abuse of notations, we still denote these approximation candidates and the cutoff functions with $v_k$ and $\varphi_k$, respectively.
\par
To show the local energy inequality \eqref{v_local_energy_inequality}, note that $v_k$ satisfies the following equation
\begin{equation} \label{localized_eq}
  \begin{split}
    -\Delta v_k + (\tilde{u}_k \cdot \nabla) v_k + \nabla (p_k \varphi_k) 
    =& f \varphi_k - 2\nabla u_k \cdot \nabla \varphi_k - u_k \Delta \varphi_k + p_k \nabla \varphi_k \\
    &+ u_k (\tilde{u}_k \cdot \nabla) \varphi_k.
  \end{split}
\end{equation}
Define $\tilde{v}_k := \chi_k * v_k$ and $\tilde{w}_k := \chi_k * w_k$. Then for any $\phi \in C_c^{\infty}(\R^6)$, testing \eqref{localized_eq} with $v_k \phi$ yields
\[ 
  \begin{split}
    2\int_{\R^6} |\nabla v_k|^2 \phi dx \leq& \int_{\R^6} |v_k|^2 |\Delta \phi| dx
      + \int_{\R^6} |v_k|^2 (\tilde{v}_k \cdot \nabla) \phi dx + \int_{\R^6} p_k \varphi_k (v_k \cdot \nabla) \phi dx  \\
    & + \int_{\R^6} f\varphi_k v_k \phi dx + \int_{\R^6} |v_k|^2 (\tilde{w}_k \cdot \nabla) \phi dx + \int_{\R^6} p_k \varphi_k \phi (u_k \cdot \nabla) \varphi_k dx \\
    & + \int_{\R^6} \big[ -2\nabla u_k \cdot \nabla \varphi_k - u_k \Delta \varphi_k + p_k \nabla \varphi_k + u_k (\tilde{u}_k \cdot \nabla) \varphi_k \big] v_k \phi dx.
  \end{split}
\]
Note that $\phi$ has compact support and $r_k \rightarrow \infty$, thus there exists $k_0 \in \N$ depending on $\{r_k\}_{k \in \N}$ and $\phi$ such that for $k \geq k_0$, we have $\nabla \varphi_k = 0$ in $\supp \phi$ and $\varphi_k = 1, 1-\varphi_k = 0$ in $\supp \phi$. Therefore, every term involing $1-\varphi_k$ or the derivative of $\varphi_k$ vanishes. This gives the inequality \eqref{v_local_energy_inequality} for $k \geq k_0$ and concludes this proof.
\end{proof}
\par

As a result of \Cref{preparation_existence}, \Cref{tightness} and concentration-compactness lemma by Lions in \cite{lions1985concentration}, we can deduce the following weak convergence of measures,
\begin{equation} \label{dirac_measures}
  \begin{split}
  \mu_k := |\nabla (v_k-u)|^2 dx \rightarrow \mu \quad &\text{weakly}-*, \\
  \nu_k := |v_k-u|^3 dx \rightarrow \nu \quad &\text{weakly}-*.
  \end{split}
\end{equation}
And we have
\begin{equation} \label{lions_condition_1}
  \begin{split}
  \mu &\geq \sum_{j \in J} \mu^{j} \delta_{x^{j}}, \\
  \nu &= \sum_{j \in J} \nu^{j} \delta_{x^{j}},
  \end{split}
\end{equation}
for some family $\{\mu^j\}_{j \in J} \subset \R^+$, $\{x^j\}_{j \in J} \subset \R^6$ and an index set $J$ which is at most countable with
\begin{equation} \label{lions_condition_2}
  C_s (\nu^j)^{2/3} \leq \mu^j \quad \text{for all } j \in J.
\end{equation}
For Lions's concentration-compactness lemma, we also suggest an expository reference, namely page 44 in Struwe's book \cite{struwe1990variational}.

\subsection{Weak solution sets and local energy estimates}

Now we can give a definition to weak solution set which contains the concentration measures $\mu$ and $\nu$ in \eqref{dirac_measures}.
\begin{definition} \label{weaksolutionset}
The quadruple $(u,p,\mu,\nu)$ is called a weak solution set of the stationary Navier-Stokes equations \eqref{sta_navierstokes} if
\begin{enumerate}
    \item $u$ and $p$ are obtained as weak limits of weak solutions $\{(u_k,p_k)\}_{k \in \N}$ of the regularized stationary Navier-Stokes equations $\eqref{reg_sta_navierstokes}$, which are uniformly bounded in $H^1(\R^6)$ and $L^{3/2}(\R^6)$ respectively.
    \item $\mu$ and $\nu$ are obtained as weak limits in \eqref{dirac_measures} in the sense of measures.
\end{enumerate}
\end{definition}

\begin{remark} \label{remark_distribution_solution}
In the weak solution set, $(u,p)$ solves the stationary Navier-Stokes equations \eqref{sta_navierstokes} in the sense of distributions. It suffices to verify
\[ \int_{\R^6} u_k (\tilde{u}_k \cdot \nabla) \psi dx \rightarrow \int_{\R^6} u (u \cdot \nabla) \psi dx. \]
Indeed, $u_k \rightarrow u$ and $\tilde{u}_k \rightarrow u$ in $L^2(\supp \psi)$ up to a subsequence as $k \rightarrow \infty$, by Rellich's compact embedding. The convergence of other linear terms is straightforward.
\end{remark}

\begin{remark} \label{remark_existence}
Given $f \in L^{3/2}(\R^6)$, the stationary Navier-Stokes equations \eqref{sta_navierstokes} have at least one weak solution set. This is a direct consequence of \Cref{preparation_existence} and \Cref{tightness}.
\end{remark}

To prove weak solutions sets satisfy certain local energy estimates, we need to study the measures $|\nabla v_k|^2dx, |v_k|^3dx$ and $v_kp_kdx$ in local scale. The goal is to show that possible concentration phenomena of these measures can be controlled by the concentration measures. For $|\nabla u_k|^2dx$ and $|u_k|^3dx$, we prove the following result.

\begin{lemma} \label{measure_convergence_1}
Suppose $\{(u_k, p_k)\}_{k \in \N}$ are solutions of regularized equations \eqref{reg_sta_navierstokes} and $(u,p,\mu,\nu)$ is the corresponding weak solution set, then
\[ |\nabla v_k|^2 dx \rightarrow |\nabla u|^2 dx + \mu \quad \text{weakly in the sense of measures} \]
and
\[ \limsup_{k \rightarrow \infty} \int_{\R^6} \psi |v_k|^3 dx \leq \int_{\R^6} \psi |u|^3dx + \int_{\R^6} \psi d\nu \]
for any $0 \leq \psi \in L^{\infty} \cap C^0(\R^6)$. Here, $\{v_k\}_{k \in \N}$ is as in \Cref{tightness}.
\end{lemma}

\begin{proof}
For any $\psi \in L^{\infty} \cap C^0(\R^6)$, we have
\begin{align}
  \int_{\R^6} |\nabla v_k|^2 \psi dx
    \rightarrow & \int_{\R^6} \psi d\mu + \int_{\R^6} |\nabla u|^2 \psi dx \quad \text{as } k \rightarrow \infty.
\end{align}
Here, since $v_k-u$ converging to zero weakly in $\dot{H}^1$, the cross integral term of $\nabla (v_k-u) \cdot \nabla u$ converges to zero. For the measure $|v_k|^3dx$,
\begin{align} \label{measure_convergence_1_eq1}
  \int_{\R^6} |v_k|^3 \psi dx \rightarrow{} \int_{\R^6} \psi d\nu + \int_{\R^6} |u|^3 \psi dx \quad \text{as } k \rightarrow \infty.
\end{align}
Here, the cross integral terms containing $(v_k-u)u$ converge to zero by Vitalli's convergence theorem, since the cross terms are uniformly integrable with respect to the measure $\psi dx$.
\end{proof}

Estimating the concentration for the measures $v_kp_kdx$ is more difficult. We need to localize the Poisson equation for the pressure and use elliptic regularity theory.

\begin{lemma} \label{measure_convergence_2}
Suppose $\{(u_k, p_k)\}_{k \in \N}$ are solutions of the regularized equations \eqref{reg_sta_navierstokes} and $(u,p,\mu,\nu)$ is the corresponding weak solution set. Let $\{v_k\}_{k \in \N}$ is as in \Cref{tightness}. Then
\[
 \limsup_{k \rightarrow \infty} \int_{\R^6} \zeta \big| v_k(p_k-\gamma) - u(p-\gamma) \big| dx
    \leq 2 \int_{\R^6} \zeta d\nu
\]
for any $0 \leq \zeta \in C_c^{\infty}(\R^6)$ and any $\gamma \in \R$. 
\end{lemma}

\begin{proof}
Note that $p_k$ satisfies the following elliptic equation,
\[ - \Delta p_k = \partial_i \partial_j (u_k^iu_k^j). \]
Indeed, this equation holds in the sense of distributions, as mentioned in \Cref{remark_distribution_solution}. Then we localize this equation with an arbitrary cut-off function $\psi \in C_c^{\infty}(\R^6)$, i.e.
\[
  \begin{split}
  - \Delta (p_k \psi) =& \psi \partial_i \partial_j (u_k^i u_k^j) - p_k \Delta \psi - 2\nabla p_k \cdot \nabla \psi \\
    =& \partial_i \partial_j (\psi u_k^i u_k^j) - u_k^i u_k^j \partial_i \partial_j \psi - \partial_i (u_k^i u_k^j) \partial_j \psi - \partial_j (u_k^i u_k^j) \partial_i \psi \\
      & - p_k \Delta \psi - 2\nabla p_k \cdot \nabla \psi.
  \end{split}
\]
\par

Now we can decompose the pressure $p_k=p_k^1+p_k^2+p_k^3+p_k^4$ with
\[
  \begin{split}
  - \Delta (p^1_k \psi) &= \partial_i \partial_j (\psi u_k^i u_k^j), \\
  - \Delta (p^2_k \psi) &= - u_k^i u_k^j \partial_i \partial_j \psi, \\
  - \Delta (p^3_k \psi) &= - \partial_i (u_k^i u_k^j) \partial_j \psi 
  						   - \partial_j (u_k^i u_k^j) \partial_i \psi, \\
  - \Delta (p^4_k \psi) &= - p_k \Delta \psi - 2\nabla p_k \cdot \nabla \psi.
  \end{split}
\]
Next we argue this decomposition exists. $p_k^1 \psi$ can be obtained by the Riesz transformation and Calderon-Zygmund estimate yields
\begin{equation} \label{CZ_pressure1}
  \|p^1_k \psi\|_{L^{3/2}} \leq \|\psi u_k^i u_k^j\|_{L^{3/2}}.
\end{equation}
Since $\nu_k \rightarrow \nu$ weakly, we have
\[
  \begin{split}
    \limsup_{k \rightarrow \infty} \int_{\R^6} |(p^1_k-p^1) \psi|^{3/2} dx
    \leq& \limsup_{k \rightarrow \infty} \int_{\R^6} |\psi (u_k^i u_k^j - u^iu^j)|^{3/2} dx \\
    \leq& \limsup_{k \rightarrow \infty} \int_{\R^6} |\psi u_k(u_k-u)|^{3/2} dx \\
      &+ \limsup_{k \rightarrow \infty} \int_{\R^6} |\psi u(u_k-u)|^{3/2} dx \\
    \leq& (1+\sqrt{2}) \limsup_{k \rightarrow \infty} \int_{\R^6} |\psi u(u_k-u)|^{3/2} dx \\
      &+ \sqrt{2} \limsup_{k \rightarrow \infty} \int_{\R^6} |\psi|^{3/2} |(u_k-u)|^3 dx \\
    \leq& \sqrt{2} \int_{\R^6} |\psi|^{3/2} d\nu.
  \end{split}
\]
By Vitalli's convergence theorem, the term in the fifth line converges to zero because $|u(u_k-u)|^{3/2}$ is uniformly integrable with respect to the finite measure $\psi^{3/2}dx$. Here, we use $|v_k|^{3/2} \leq \sqrt{2}|v_k-u|^{3/2} + \sqrt{2}|u|^{3/2}$ and $u_k=v_k$ in $B_{r_k}$ with $r_k \rightarrow \infty$ from \Cref{tightness}.
\par
Since $(u_k^i u_k^j) \partial_i \partial_j \psi$ is bounded in $L^{3/2}$ uniformly in $k \in \N$, $p_k^2 \psi$ can be obtained by convolution with singular kernels and Calderon-Zygmund estimate yields
\begin{equation} \label{CZ_pressure2}
  \|\nabla^2 (p^2_k \psi) \|_{L^q} \leq \|u_k^i u_k^j \partial_i \partial_j \psi\|_{L^q} \quad \text{for any } q \in (1,3/2].
\end{equation}
Thus by compact Sobolev embedding and $u_k^i u_k^j - u^i u^j \rightarrow 0$ in $L^s$ for any $s \in [1,3/2)$, we can deduce
\[ \limsup_{k \rightarrow \infty} \int_{\R^6} |(p^2_k-p^2) \psi|^{3/2} dx = 0. \]
\par

Similarly, for $p_k^3 \psi$ we have
\begin{equation} \label{CZ_pressure3}
  \|\nabla^2 (p^3_k \psi) \|_{L^q} \leq \|-\partial_i (u_k^i u_k^j) \partial_j \psi - \partial_j (u_k^i u_k^j) \partial_i \psi\|_{L^q} \quad \text{for any } q \in (1,6/5].
\end{equation}
Again by compact Sobolev embedding and $-\partial_i (u_k^i u_k^j) \partial_j \psi - \partial_j (u_k^i u_k^j) \partial_i \psi \rightarrow 0$ in $L^s(\supp \psi)$ for any $s \in [1,6/5)$, we can deduce
\[ \limsup_{k \rightarrow \infty} \int_{\R^6} |(p^3_k-p^3) \psi|^{3/2} dx = 0. \]

\par

Let $p_k^4:=p_k-(p_k^1+p_k^2+p_k^3)$ and $p_k^4$ solves the corresponding equation in the sense of distribution. We have
\begin{equation} \label{CZ_pressure4}
  \|\nabla (p^4_k \psi) \|_{L^q} \leq \|p_k \Delta \psi + 2\nabla p_k \cdot \nabla \psi\|_{L^q} \quad \text{for any } q \in (1,6/5].
\end{equation}
By compact Sobolev embedding and $\nabla p_k - \nabla p \rightarrow 0$ in $L^s(\supp \psi)$ for any $s \in [1,6/5)$, we can deduce
\[ \limsup_{k \rightarrow \infty} \int_{\R^6} |(p^4_k-p^4) \psi|^{3/2} dx = 0. \]
\par

Now we combine the estimates \eqref{CZ_pressure1}, \eqref{CZ_pressure2}, \eqref{CZ_pressure3} and \eqref{CZ_pressure4} and get
\begin{align*}
  \limsup_{k \rightarrow \infty} \int_{\R^6} |(p_k-p) \psi|^{3/2} dx 
    \leq \sqrt{2} \int_{\R^6} |\psi|^{3/2} d\nu
\end{align*}
Therefore, we can choose $\psi = \zeta^{2/3}$ and bound the concentration of the measure as following
\[ 
  \begin{split}
  \limsup_{k \rightarrow \infty} \int_{\R^6}& \zeta \big| v_k(p_k-\gamma) - u(p-\gamma) \big|dx \\
  	\leq& \limsup_{k \rightarrow \infty} \int_{\R^6} \zeta \big| v_k(p_k-\gamma) - u(p-\gamma) \big|dx \\
    \leq& \limsup_{k \rightarrow \infty} \int_{\R^6} \zeta |v_k||p_k - p| dx
      + \limsup_{k \rightarrow \infty} \int_{\R^6} \zeta |v_k- u||p-\gamma| dx \\
    \leq& \limsup_{k \rightarrow \infty} \int_{\R^6} \zeta |v_k-u||p_k - p| dx
      + \limsup_{k \rightarrow \infty} \int_{\R^6} \zeta |u||p_k - p| dx \\
    \leq& 2^{1/3} \int_{\R^6} \zeta d\nu
  \end{split}
\]
Due to Vitalli's convergence theorem, the second term in third line and the secind term in fourth line converge to zero. Note that $\zeta$ is nonnegative, so $\psi$ is a Lipschitz function by Corollary 5.1 in \cite{wu2021partially}.
\end{proof}

Now we can prove local energy estimates for the weak soution set $(u,p,\mu,\nu)$. We prove two variants which are similar in nature. We need two variants only for technical reasons in the proof of partial regularity results.

\begin{proposition} \label{measure_convergence_3}
Let the assumptions be as in \Cref{measure_convergence_1}, then the following local energy inequalities hold
\begin{equation} \label{local_energy_1}
  \begin{split}
  2\int_{\R^6} |\nabla u|^2 \phi dx + 2\int_{\R^6} \phi d\mu 
    \leq \int_{\R^6} |u|^2 |\Delta \phi| dx
      + \sum_{i=1}^n \int_{\R^6} |u|^3 |\nabla \phi_i| dx + 2\sum_{i=1}^n \int_{\R^6} |\nabla \phi_i| d\nu& \\
      + \sum_{i=1}^n \int_{\R^6} |\nabla \phi_i| |p-\gamma_i|^{3/2} dx + \int_{\R^6} f \cdot u \phi dx&
  \end{split}
\end{equation}
  
  \begin{equation} \label{local_energy_2}
    \begin{split}
    2\int_{\R^6} |\nabla u|^2 \phi dx + 2\int_{\R^6} \phi d\mu 
      \leq \int_{\R^6} |u|^2 |\Delta \phi| dx
        + 2\int_{\R^6} |u-\overline{u}_{\phi}|^3 |\nabla \phi| dx + 3\int_{\R^6} |\nabla \phi| d\nu& \\
      + \int_{\R^6} |u|^2 (u \cdot \nabla) \phi dx + \int_{\R^6} p (u \cdot \nabla) \phi dx + \int_{\R^6} f \cdot u \phi dx&
    \end{split}
  \end{equation}
 for any $n \geq 1$, any $\{\gamma_i\}_{1 \leq i \leq n} \subset \R$, any cut-off functions $\phi, \{\phi_i\}_{1 \leq i \leq n} \subset C_c^{\infty}(\R^6)$ with $\phi = \sum_{i=1}^n \phi_i$, where
  \[ \overline{u}_{\phi} = \frac{1}{\mathcal{L}(\supp \phi)} \int_{\supp \phi} u(x) dx. \]
\end{proposition}

\begin{proof}
Since $u_k=v_k$ in $B_{r_k}$ with $r_k \rightarrow \infty$, we do not distinguish the difference of $u_k$ and $v_k$ in any ball of finite radius when passing to limit. \par
For the first local energy inequality \eqref{local_energy_1}, we pass \eqref{v_local_energy_inequality} to the limit by the weak convergence of measure and the strong convergence of $\{v_k\}_{k \in \N}$ in $L^2(\supp \phi)$. In the limiting process, we also need to combine the control on other measures obtained in \Cref{measure_convergence_1} and \Cref{measure_convergence_2}. The first subtle point is the term involving the cubic velocity. Note that
\[ 
  \begin{split}
  \int_{\R^6} |\tilde{v}_k|^3 |\nabla \phi| dx 
    =& \int_{\R^6} \Big| \int_{\R^6} v_k(x-y) \chi_k(y) |\nabla \phi(x)|^{1/3} dy \Big|^3 dx \\
    =& \| h_1 + h_2 \|_{L^3}^3,
  \end{split}
\]
where
\[
  \begin{split}
  h_1 &= \int_{\R^6} v_k(x-y) \chi_k(y) \big( |\nabla \phi(x)|^{1/3} - |\nabla \phi(x-y)|^{1/3} \big) dy, \\
  h_2 &= \int_{\R^6} v_k(x-y) \chi_k(y) |\nabla \phi(x-y)|^{1/3} dy. \\
  \end{split}
\]
For $h_1$, notice that $\diam (\supp \chi_k) \rightarrow 0$ and $x \rightarrow |x|^{1/3}$ is $1/3-$H{\"o}lder continuous, then Young's inequality for convolution yields
\[
  \begin{split}
  \|h_1\|_{L^3} &\leq 
     \Big\| \int_{\R^6} v_k(x-y) \chi_k(y) \frac{|\nabla \phi(x)|^{1/3} 
      - |\nabla \phi(x-y)|^{1/3}}{|y|^{1/3}} \diam (\supp \chi_k)^{1/3} dy \Big\|_{L^3} \\
    &\leq C \diam (\supp \chi_k)^{1/3} \|\phi\|_{C^2} \|\tilde{v}_k\|_{L^3} \\
    &\leq C \diam (\supp \chi_k)^{1/3} \|\phi\|_{C^2} \|v_k\|_{L^3}.
  \end{split}
\]
Since $v_k$ is uniformly bounded in $L^3$, the term $h_1$ converges to zero in $L^3$ when passing to limit.
For $h_2$, Young's inequality for convolution yields
\[
  \begin{split}
  \|h_2\|_{L^3} &= \| (v_k|\nabla \phi|^{1/3}) * \chi_k \|_{L^3} \\
    &\leq \| v_k|\nabla \phi|^{1/3} \|_{L^3}.
  \end{split}
\]
Combining these estimates and using the boundedness of $\{u_k\}_{k \in \N}$, we have
\[
  \begin{split}
  \limsup_{k \rightarrow \infty} \int_{\R^6} |v_k|^2 (\tilde{v}_k \cdot \nabla) \phi dx
    \leq& \frac{2}{3} \limsup_{k \rightarrow \infty} \int_{\R^6} |v_k|^3 |\nabla \phi| dx
      + \frac{1}{3} \limsup_{k \rightarrow \infty} \int_{\R^6} |\tilde{v}_k|^3 |\nabla \phi| dx \\
    \leq& \int_{\R^6} \psi |u|^3dx + \int_{\R^6} \psi d\nu.
  \end{split}
\]
\par
Now we move to the term involving pressure. Because $u_k$ is weakly divergence-free and $v_k = u_k$ in $\supp \phi$ for large $k$, we have
\[ 
  \begin{split}
  \limsup_{k \rightarrow \infty}\int_{\R^6} p_k (v_k \cdot \nabla) \phi dx 
    =& \sum_{i=1}^n \limsup_{k \rightarrow \infty} \int_{\R^6} p_k u_k \cdot \nabla \phi_i dx \\
    =& \sum_{i=1}^n \limsup_{k \rightarrow \infty} \int_{\R^6} (p_k-\gamma_i) u_k \cdot \nabla \phi_i dx \\
    \leq& \sum_{i=1}^n \int_{\R^6} |u|^3 |\nabla \phi_i| dx + 2\sum_{i=1}^n \int_{\R^6} |\nabla \phi_i| d\nu \\
    &+ \sum_{i=1}^n \int_{\R^6} |\nabla \phi_i| |p-\gamma_i|^{3/2} dx.
  \end{split}
\]
The inequality follows from \Cref{measure_convergence_2}. For the second local energy inequality \eqref{local_energy_2}, the subtle point is the cubic term of $u$. Note that $u_k,\tilde{u}_k$ and $u$ are weakly divergence-free. For the first line of \eqref{v_local_energy_inequality}, observe that
\begin{equation} \label{measure_convergence_3_eq1}
  \begin{split}
  \int_{\R^6} |v_k|^2 (\tilde{v}_k \cdot \nabla) \phi dx
    =& \int_{\R^6} |u_k - \overline{u}_{k,\phi} + \overline{u}_{k,\phi}|^2 (\tilde{u}_k \cdot \nabla) \phi dx \\
    =& \int_{\R^6} \big[ |u_k-\overline{u}_{k,\phi}|^2 + 2(u_k-\overline{u}_{k,\phi}) \cdot \overline{u}_{k,\phi} \big] 
      (\tilde{u}_k \cdot \nabla) \phi dx \\
    =& \int_{\R^6} |u_k - \overline{u}_{k,\phi}|^2 \big[ (\tilde{u}_k-\overline{u}_{k,\phi}) \cdot \nabla \big] \phi dx \\
      &+ \int_{\R^6} |u_k - \overline{u}_{k,\phi}|^2 (\overline{u}_{k,\phi} \cdot \nabla) \phi dx \\
    &+ 2\int_{\R^6} \big[(u_k-\overline{u}_{k,\phi}) \cdot \overline{u}_{k,\phi} \big] 
      (\tilde{u}_k \cdot \nabla) \phi dx.
  \end{split}
\end{equation}
Next, similar to \eqref{measure_convergence_1_eq1}, we argue that the individual terms above can be bounded by the same form of the weak limit $u$ and the concentration mass $\nu$. For the third line of \eqref{measure_convergence_3_eq1}, since $\tilde{u}_k-\overline{u}_{k,\phi} = (u_k - \overline{u}_{k,\phi}) * \chi_k$, we can apply the same trick as above and use Young's inequality for convolution, therefore it is sufficient to look at the following term
\[ 
  \begin{split}
  \int_{\R^6} |u_k - \overline{u}_{k,\phi}|^2 &\big[ (u_k-\overline{u}_{k,\phi}) \cdot \nabla \big] \phi dx \\
    \leq& \int_{\R^6} |u_k - \overline{u}_{k,\phi} - (u-\overline{u}_{\phi}) |^3 |\nabla \phi| dx + \int_{\R^6} | u-\overline{u}_{\phi} |^3 |\nabla \phi| dx \\
    & + \int_{\R^6} 3|u_k - \overline{u}_{k,\phi} - (u-\overline{u}_{\phi}) |^2 |u-\overline{u}_{\phi}| |\nabla \phi| dx \\
    & + \int_{\R^6} 3|u_k - \overline{u}_{k,\phi} - (u-\overline{u}_{\phi}) | |u-\overline{u}_{\phi}|^2 |\nabla \phi| dx \\
    \leq& \int_{\R^6} |u_k - u - (\overline{u}_{k,\phi}-\overline{u}_{\phi}) |^3 |\nabla \phi| dx + \int_{\R^6} | u-\overline{u}_{\phi} |^3 |\nabla \phi| dx \\
    \rightarrow& \int_{\R^6} |\nabla \phi| d\nu + \int_{\R^6} | u-\overline{u}_{\phi} |^3 |\nabla \phi| dx
    \quad \text{as } k \rightarrow \infty.
  \end{split}
\]
Since $u_k \rightarrow u$ strongly in $L^1(\supp \phi)$ up to a subsequence, $\overline{u}_{k,\phi}-\overline{u}_{\phi} \rightarrow 0$ when $k \rightarrow \infty$. Because $u_k \rightarrow u$ and $\tilde{u}_k \rightarrow u$ locally strongly in $L^2$, we can pass to limit for the last two terms in the last line of \eqref{measure_convergence_3_eq1}. Hence we have
\[ 
  \begin{split}
  \limsup_{k \rightarrow \infty} &\int_{\R^6} |v_k|^2 (\tilde{v}_k \cdot \nabla) \phi dx
    \leq \int_{\R^6} |\nabla \phi| d\nu + \int_{\R^6} | u-\overline{u}_{\phi} |^3 |\nabla \phi| dx \\
    &+ \int_{\R^6} |u - \overline{u}_{\phi}|^2 (\overline{u}_{\phi} \cdot \nabla) \phi dx
    + 2\int_{\R^6} \big[(u-\overline{u}_{\phi}) \cdot \overline{u}_{\phi} \big] 
      (u \cdot \nabla) \phi dx \\
    =& \int_{\R^6} |\nabla \phi| d\nu + \int_{\R^6} | u-\overline{u}_{\phi} |^3 |\nabla \phi| dx \\
    &+ \int_{\R^6} |u|^2 (u \cdot \nabla) \phi dx - \int_{\R^6} |u - \overline{u}_{\phi}|^2 \big[ (u-\overline{u}_{\phi}) \cdot \nabla \big] \phi dx \\
    \leq& \int_{\R^6} |\nabla \phi| d\nu + 2\int_{\R^6} | u-\overline{u}_{\phi} |^3 |\nabla \phi| dx
    + \int_{\R^6} |u|^2 (u \cdot \nabla) \phi dx.
  \end{split}
\]
\par

For the term involving pressure. \Cref{measure_convergence_2} yields
\[ 
  \begin{split}
  \limsup_{k \rightarrow \infty}\int_{\R^6} p_k (v_k \cdot \nabla) \phi dx 
    \leq & \int_{\R^6} p (u \cdot \nabla) \phi dx + 2 \int_{\R^6} |\nabla \phi| d\nu.
  \end{split}
\]
\par

Finally, the two inequalities follow from the local energy inequalities for the regularized Navier-Stokes equations and the estimates above.
\end{proof}


\section{Partial regularity theory} \label{section_partial_regularity}

Partial regularity theory is a deep result of the scaling invariance and local energy estimates. In this section, we adapt Caffarelli, Kohn and Nirenberg’s argument \cite{caffarelli1982partial} and Lin's argument \cite{lin1998new} to the stationary Navier-Stokes equations in $\R^6$ with the presence of concentration measures.

\subsection{Scaling invariance}
Given a smooth solution $(u,p)$ of the stationary Navier-Stokes equations with force $f$, $(u_r,p_r)_{r>0}$ define by
\[ u_r(x) := ru(rx) \quad p_r(x) := r^2p(rx) \]
solves the equations \eqref{sta_navierstokes} with the scaled force
\[ f_r(x) := r^3 f(rx). \]
This scaling property also holds for weak solution sets.

\begin{lemma} \label{scale_invariance_sta}
If $(u,p,\mu,\nu)$ is a weak solution set of the stationary Navier-Stokes equations \eqref{sta_navierstokes} with external force $f$, then for any $r>0$, the scaled quadruple $(u_r,p_r,\mu_r,\nu_r)$ is also a weak solution set of \eqref{sta_navierstokes} with external force $f_r$, where $u_r,p_r$ and $f_r$ are defined as above and $\mu_r,\nu_r$ are defined as
\[ d\mu_r(x) := r^{-2} d\mu(rx)  \quad d\nu_r(x) :=r^{-3} d\nu(rx). \]
\end{lemma}

Due to the scaling invariance of the stationary Navier-Stokes equations in $\R^6$, we can define the following quantities which do not change after rescaling. We call such quantities dimensionless. The estimates in this section are based on these quantities.
\[
	\begin{split}
	A(x_0,r) &:= r^{-2} \int_{B_r(x_0)} |\nabla u|^2 dx, \\
	A_c(x_0,r) &:= r^{-2} \int_{B_r(x_0)} d\mu, \\
	D(x_0,r) &:= r^{-4} \int_{B_r(x_0)} |u|^2 dx, \\
	G(x_0,r) &:= r^{-3} \int_{B_r(x_0)} |u|^3 dx, \\
	G_c(x_0,r) &:= r^{-3} \int_{B_r(x_0)} d\nu, \\
	H(x_0,r) &:= r^{-3} \int_{B_r(x_0)} |u - \overline{u}_{x_0,r}|^3 dx, \\
	K(x_0,r) &:= r^{-3} \int_{B_r(x_0)} |p|^3 dx, \\
	L(x_0,r) &:= r^{-3} \int_{B_r(x_0)} |p - \overline{p}_{x_0,r}|^3 dx, \\
	F(x_0,r) &:= \int_{B_r(x_0)} |f|^2 dx,
	\end{split}
\]
where
\[ 
	\begin{split}
	\overline{u}_{x_0,r} = \frac{1}{\mathcal{L}(B_r)} \int_{B_r(x_0)} u(x) dx, \\
	\overline{p}_{x_0,r} = \frac{1}{\mathcal{L}(B_r)} \int_{B_r(x_0)} p(x) dx.
	\end{split}
\]

\subsection{Partial regularity results}

The main result in this section is the following.

\begin{theorem} \label{mainregularity}
Suppose $(u,p,\nu,\mu)$ is a weak solution set of the stationary Navier-Stokes equations \eqref{sta_navierstokes} in $\R^6$ with a weakly solenoidal force $f \in L_{\text{loc}}^q(\R^6), q>3$, then its singular set $S$ satifies $\mathcal{H}^2(S) < C$ for some constant $C>0$ depending on $f$.
\end{theorem}

For preparations, we need to establish three technical lemmas to estimate these dimensionless quantities. The estimates in the rest of this section are uniform in $x_0$, thus we take $x_0=0$ and omit $x_0$ in most proofs. Also, without special remarks, all constants are absolute.

\begin{lemma} \label{Sobolev_type_inequality}
Suppose $(u,p,\nu,\mu)$ is a weak solution set of the stationary Navier-Stokes equations \eqref{sta_navierstokes} in $\R^6$, then for any $x_0 \in \R^6$ and $r>0$,
\[
	\begin{split}
	G(x_0,r) \leq& C_1 A^{3/2}(x_0,r) + C_1 D^{3/2}(x_0,r), \\
	G_c(x_0,r) \leq& C_1 A_c^{3/2}(x_0,r), \\
	H(x_0,r) \leq& C_1 A^{3/2}(x_0,r).
	\end{split}
\]
\end{lemma}

\begin{proof}
The estimates for $G$ and $H$ are due to Sobolev inequality and Sobolev-Poincar\'{e} inequality. For the quantity $G_c(x_0,r)$, it suffices to prove for any countably family $J' \subset J$,
\[ r^{-3} \sum_{i \in J'} \nu^j \leq C_1 \Big( r^{-2} \sum_{j \in J'} \mu^j \Big)^{3/2}. \]
To verify this inequality, we observe that \eqref{lions_condition_2} implies
\[ C_s \sum_{i \in J'} (\nu^j)^{2/3} \leq \sum_{i \in J'} \mu^j. \]
Since $\|\cdot\|_{l^{\alpha}} \leq \|\cdot\|_{l^{\beta}}$ for any $\alpha > \beta\geq 1$, we can deduce
\[ \Big( \sum_{i \in J'} \nu^j \Big)^{2/3} \leq C_1 \sum_{j \in J'} \mu^j. \]
\end{proof}

\begin{lemma} \label{pressure_estimate_1}
Suppose $(u,p,\nu,\mu)$ is a weak solution set of the stationary Navier-Stokes equations \eqref{sta_navierstokes} in $\R^6$, then for any $x_0 \in \R^6$, $\rho > 0$ and $r \in (0, \frac{\rho}{2}]$, we have
\[
	\begin{split}
	L(x_0,r) \leq C_2 G(x_0,2r) + C_2 r^{9/2} \Big( \int_{2r<|y|<\rho} \frac{|u|^2}{|y|^7} dy \Big)^{3/2} &\\
	+ C\Big(\frac{r}{\rho}\Big)^{9/2} [L(x_0,\rho) + G(x_0,\rho)]. &
	\end{split}
\]

\end{lemma}

\begin{proof}
For simplicity, we prove the estimate for $x_0=0$. Choose a cut-off function $\psi \in C_c^{\infty}$ such that
\[
	\begin{split}
	\psi &= 1 \quad \text{in} \quad B_{3\rho/4}, \\
	\psi &= 0 \quad \text{in} \quad \R^6 \backslash B_{\rho}, \\
	|\nabla \psi| &\leq C_2 \rho^{-1}, \\
	|\nabla^2 \psi| &\leq C_2 \rho^{-2}.
	\end{split}
\]
then we localize the pressure equation and integrate by parts to move the differentiation from $u$ and $p$ to $\psi$,
\[
	\begin{split}
	p(x)\psi(x) &= (-\Delta)^{-1} (-\Delta) (p\psi) (x) \\
		=& C_2 \int_{\R^6} \frac{1}{|x-y|^4} \big[ \psi \partial_i \partial_j (u_iu_j) - 2 \nabla \psi \cdot \nabla p - p \Delta \psi \big] dy \\
		=& C_2 \int_{\R^6} \partial_i \partial_j \Big(\frac{1}{|x-y|^4}\Big) \psi u_iu_j dy + C_2 \int_{\R^6} u_i u_j \frac{\partial_i \partial_j \psi} {|x-y|^4} dy \\
		&+ C_2 \int_{\R^6} u_i u_j \frac{(x_i-y_j)\partial_i \psi}{|x-y|^6} dy + C_2 \int_{\R^6} p \Big( \frac{\Delta \psi}{|x-y|^4} + \frac{\nabla \psi \cdot (x-y)}{|x-y|^6} \Big) dy \\
		=& p_1 + p_2 + p_3 + p_4,
	\end{split}
\]
where
\[
	\begin{split}
	p_1 =& C_2 \int_{B_{2r}} \psi u_iu_j \partial_i \partial_j \Big(\frac{1}{|x-y|^4}\Big) dy, \\
	p_2 =& C_2 \int_{2r<|y|<\rho} \psi u_iu_j \partial_i \partial_j \Big(\frac{1}{|x-y|^4}\Big) dy, \\
	p_3 =& C_2 \int_{\R^6} u_i u_j \frac{(x_i-y_j)\partial_i \psi}{|x-y|^6} dy, \\
	p_4 =& C_2 \int_{\R^6} p \Big( \frac{\Delta \psi}{|x-y|^4} + \frac{\nabla \psi \cdot (x-y)}{|x-y|^6} \Big) dy.
	\end{split}
\]
\par

For $p_1$, Calderon-Zygmund estimate yields
\[ \int_{B_r} |p_1|^{3/2} dy \leq C_2 \int_{B_{2r}} |u|^3 dy. \]
\par

For $p_2,p_3$ and $p_4$, we estimate the $L^{\infty}-$norm of their gradients
\[
	\begin{split}
	|\nabla p_2| &\leq \int_{2r<|y|<\rho} \frac{|u|^2}{|y|^7} dy, \\
	|\nabla p_3| &\leq \rho^{-7} \int_{B_{\rho} \backslash B_{3\rho/4}} |u|^2 dy, \\
	|\nabla p_4| &\leq \rho^{-7} \int_{B_{\rho} \backslash B_{3\rho/4}} |p| dy.
	\end{split}
\]
The $L^{\infty}$ bound of $\nabla p_2$ uses the fact $2|x-y|>|y|$ when $x \in B_r$ and $y \in \R^6 \backslash B_{2r}$. For $\nabla p_3$ and $\nabla p_4$, note that $\nabla \psi = 0$ in $B_{3\rho/4}$ and $|x-y|>\frac{\rho}{4}$. \par
Integrating $p_2$ in $B_r$ gives
\[
	\begin{split}
	r^{-3} \int_{B_r} |p_2-\overline{p}_{2,r}|^{3/2} dx
		&\leq C_2 r^3 \|p_2-\overline{p}_{2,r}\|_{L^{\infty}} \\
		&\leq C_2 r^{9/2} \Big( \int_{2r<|y|<\rho} \frac{|u|^2}{|y|^7} dy \Big)^{3/2}.
	\end{split}
\]
The estimates for $p_3$ and $p_4$ are similar and the combination of these estimates concludes the proof.
\end{proof}

\begin{lemma} \label{pressure_estimate_2}
Suppose $(u,p,\nu,\mu)$ is a weak solution set of the stationary Navier-Stokes equations \eqref{sta_navierstokes} in $\R^6$, then for any $x_0 \in \R^6, r>0$ and any $\theta \in (0,\frac{1}{2})$, we have
\[ K(x_0, \theta r) \leq C_2 \theta^{-3} A^{3/2}(x_0, r) + C_2 \theta^3 K(x_0, r). \]
\end{lemma}

\begin{proof}
This proof is similar to the proof of \Cref{pressure_estimate_1}. Choose a cut-off function $\psi \in C_c^{\infty}$ such that
\[
	\begin{split}
	\psi &= 1 \quad \text{in} \quad B_{3r/4}, \\
	\psi &= 0 \quad \text{in} \quad \R^6 \backslash B_r, \\
	|\nabla \psi| &\leq C_2 r^{-1}, \\
	|\nabla^2 \psi| &\leq C_2 r^{-2}.
	\end{split}
\]
Since $u$ is weakly divergence-free, we can write the equation for the pressure $p$ as
\[ - \Delta u = \partial_i \partial_j \big[ (u_i-\overline{u}_{i,r}) (u_j-\overline{u}_{j,r}) \big], \]
then we can localize the equation
\[
	\begin{split}
	p(x)\psi(x) &= p_1(x) + p_2(x) + p_3(x), \\
	p_1(x) &= C_2 \int_{B_r} (u_i-\tilde{u}_{i,r})(u_j-\tilde{u}_{j,r}) \psi 
		\partial_i \partial_j \Big( \frac{1}{|x-y|^4} \Big) dy, \\
	p_2(x) &= C_2 \int_{B_r} (u_i-\tilde{u}_{i,r})(u_j-\tilde{u}_{j,r}) 
		\Big( \frac{\partial_i \partial_j \psi}{|x-y|^4} + \partial_j \psi \frac{4(x_i-y_i)}{|x-y|^6} \Big) dy, \\
	p_3(x) &= C_2 \int_{B_r} p \Big( \frac{\Delta \psi}{|x-y|^4} 
										  + \frac{4(x-y) \cdot \nabla \psi}{|x-y|^6} \Big) dy.
	\end{split}
\]
\par

For $p_1$, Calderon-Zygmund theory yields
\[ \int_{B_{\theta r}} |p_1|^{3/2} dx \leq C_2 \int_{B_r} |u-\overline{u}_r|^3 dx. \]
\par

For $p_2$, note that $\nabla \psi$ is supported in $B_r \backslash B_{3r/4}$, and for any $x \in B_{\theta r}, y \in B_r \backslash B_{3r/4}$, $|x-y| > \frac{r}{4}$. It yields that
\[ |p_2| \leq C_2 r^{-6} \int_{B_r} |u-\overline{u}_r|^2 dx. \]
Integrating in $B_{\theta r}$ gives
\[
	\begin{split}
	\int_{B_{\theta r}} |p_2|^{3/2} dx
		\leq & C_2 \theta^6 r^{-3} \Big( \int_{B_r} |u-\overline{u}_r|^2 dy \Big)^{3/2} \\
		\leq & C_2 \theta^6 \int_{B_r} |u-\overline{u}_r|^3 dy.
	\end{split}
\]
\par

The estimate for $p_3$ is similar to $p_2$. Combining all the estimates and the Sobolev-Poincar\'{e} inequality yields
\[ K(\theta r) \leq C_2 \theta^{-3} A^{3/2}(r) + C_2 \theta^3 K(r) \]
then the uniformality in $x_0 \in \R^6$ concludes the proof.
\end{proof}

Next, we use a special cut-off function and related bounds. This cutoff function was introduced by Dong and Strain in \cite{dong2012partial}. The bounds follow from direct computations.
\begin{lemma} \label{cutoff}
For $x_0 \in \R^6, \rho > 0$ and $\theta \in (0,\frac{1}{2}]$, define $r_n := \theta^n\rho, n \in \N$, suppose $\phi \in C_c^{\infty}(B_{\rho}(x_0))$ satisfies
\[
	\begin{split}
	\phi = 0 \quad &\text{in } B_{\rho}(x_0) \backslash B_{\rho/3}(x_0), \\
	0 \leq \phi \leq 1 \quad &\text{in } B_{\rho}(x_0), \\
	\phi = 1 \quad &\text{in } B_{\rho/4}(x_0), \\
	|\nabla \phi| \leq C_3 \quad &\text{in } B_{\rho}(x_0) \backslash B_{\rho/2}(x_0), \\
	|\nabla^2 \phi| \leq C_3 \quad &\text{in } B_{\rho}(x_0) \backslash B_{\rho/2}(x_0)
	\end{split} 
\]
and define $\psi_n : \R^6 \rightarrow \R$ with
\[ \psi_n(x) := ( r_n^2 + |x-x_0|^2 )^{-2}, \]
then we have the following bounds
\[
	\begin{split}
	\phi \psi_n \geq C_3^{-1} r_n^{-4} \quad &\text{in } B_{r_n}(x_0), \\
	- \phi \Delta \psi_n \geq C_3^{-1} r_n^{-6} \quad &\text{in } B_{r_n}(x_0), \\
	- \phi \Delta \psi_n \geq 0 \quad &\text{in } B_{\rho}(x_0), \\
	\phi \psi_n \leq C_3 r_n^{-4} \quad &\text{in } B_{r_n}(x_0), \\
	\phi \psi_n \leq C_3 r_k^{-4} \quad &\text{in } B_{r_{k-1}}(x_0) \backslash B_{r_k}(x_0), 1 \leq k \leq n, \\
	|\psi_n \nabla \phi| + |\phi \nabla \psi_n| \leq C_3 r_n^{-5} \quad &\text{in } B_{r_n}(x_0), \\
	|\psi_n \nabla \phi| + |\phi \nabla \psi_n| \leq C_3 r_k^{-5} \quad &\text{in } B_{r_{k-1}}(x_0) \backslash B_{r_k}(x_0), 1 \leq k \leq n, \\
	|\psi_n \Delta \phi| + |\nabla \phi \cdot \nabla \psi_n| \leq C_3 \rho^{-6} \quad &\text{in } B_{\rho}(x_0),
	\end{split}
\]
where $C_3 > 0$ is an absolute constant.
\end{lemma}

Now, we prove two important local regularity results based on the smallness of certain dimensionless quantities. As an interesting byproduct, we exclude concentration phenomena of $\mu$ and $\nu$ when $G+G_c+K$ is suitably small.

\begin{proposition} \label{ckn1}
There exists absolute constants $\epsilon>0, \kappa>0$, and for any fixed $q > 3$, there exists another constant $C=C(q)$ with the following property. If a weak solution set $(u,p,\nu,\mu)$ of the Navier-Stokes equations in $\R^6$ with a weakly solenoidal force $f \in L_{\text{loc}}^q(\R^6), q>3$ satisfies
\begin{equation} \label{ckn1_initial}
	\begin{split}
	\int_{B_1(x_0)} (|u|^3 dx + d\nu) + \int_{B_1(x_0)} |p|^{3/2} dx &\leq \epsilon, \\
	\int_{B_1(x_0)} |f|^q dx &\leq \kappa,
	\end{split}
\end{equation}
then $\|u\|_{L^{\infty}(B_{1/2}(x_0))} < C $ and $\nu|_{B_1(x_0)}= \mu|_{B_1(x_0)} = 0$.
\end{proposition}

\begin{proof}[Proof of \Cref{ckn1} and \Cref{ckn1_scaled}]

Let $r_n = 2^{-n}$. The idea is to show
\begin{align}
	G_c(x',r_n) + G(x',r_n) + L(x',r_n) &\leq \epsilon^{2/3} r_n^3 \quad \forall x' \in B_{1/2}(x_0) \label{GLclaim}, \\
	A_c(x',r_n) + A(x',r_n) + D(x',r_n) &\leq C_B \epsilon^{2/3} r_n^2 \quad \forall x' \in B_{1/2}(x_0) \label{ADclaim}
\end{align}

iteratively by proving the following claims.

\begin{claim} \label{ckn1_initialclaim}
The inequality $\eqref{GLclaim}_1$ holds.
\end{claim}
\begin{claim} \label{ckn1_claim1}
$\{\eqref{GLclaim}_k\}_{1 \leq k \leq n}$ implies $\eqref{ADclaim}_{n+1}$.
\end{claim}
\begin{claim} \label{ckn1_claim2}
$\{\eqref{ADclaim}_k\}_{2 \leq k \leq n}$ implies $\eqref{GLclaim}_n$.
\end{claim}

Then we have
\begin{equation} \label{ckn1_main_estimate}
	\sup_{n \in \N, x' \in B_{1/2}(x_0)} r_n^{-6} \int_{B_{r_n}(x')} |u|^3 dx \leq \epsilon^{2/3}.
\end{equation}
This gives the uniform bound of $u$ in $B_{1/2}(x_0)$. The vanishing of concentration points is immediate by letting $n \rightarrow \infty$ around concentration points. Note that the argument can be adapted to prove \eqref{ckn1_main_estimate} for $x' \in B_{r'}(x_0)$ for any $r'<1$, thus we can exclude all concentration phenomena in the open ball $B_1$.
\par

The rest of this proof is devoted to proving the three claims.

\begin{proof}[Proof of \Cref{ckn1_initialclaim}]
This is straightforward by imposing a suitable smallness condition on $\epsilon$. 
\end{proof}

\begin{proof}[Proof of \Cref{ckn1_claim1}]
Choose the cut-off function $\phi\psi_n$ with $\rho = 1, \theta = \frac{1}{2}$ in \Cref{cutoff}. And we do the following decomposition
\[ \phi\psi_n = \phi\psi_n \eta_1 = \phi\psi_n\eta_n + \sum_{k=1}^{n-1} \phi\psi_n(\eta_k - \eta_{k+1}) \]
with cut-off functions $\{\eta_k\}_{k \in \N}$ such that
\[ 
	\begin{split}
	\eta_k =& 1 \quad \text{in} \quad B_{7r_k/8}, \\
	\eta_k =& 0 \quad \text{in} \quad \R^6 \backslash B_{r_k}, \\
	|\nabla \eta_k| \leq& Cr_k^{-1}.
	\end{split}
\]
Denote $\varphi_k := \phi\psi_n(\eta_k - \eta_{k+1})$ for $1 \leq k \leq n-1$ and $\varphi_n := \phi\psi_n\eta_n$, then
\[
	\begin{split}
	\phi\psi_n &= \sum_{k=1}^n \varphi_k, \\
	|\nabla \varphi_k| &\leq C_3 r_k^{-5}.
	\end{split}
\]
The first equality is due to the fact that $\phi\psi_n = \phi\psi_n \eta_1$, since $\phi = 0$ in $B_1 \backslash B_{1/3}$, as defined in \Cref{cutoff}. Also, the gradient bound for $\varphi$ is a simple consequence of the bounds in \Cref{cutoff}.
\par

The local energy inequality \eqref{local_energy_1} yields
\[
	\begin{split}
	- \int_{B_{1/2}} |u|^2 \phi \Delta \psi_n dx + 2 \int_{B_{1/2}} \phi \psi_n (|\nabla u|^2 dx + d\mu) \leq \int_{B_{1/2}} |u|^2 |\psi_n \Delta \phi + \nabla \phi \cdot \nabla \psi_n| dx & \\
	+ \sum_{k=1}^n \int_{B_{1/2}} |\nabla \varphi_k| (|u|^3dx + d\nu) + \sum_{k=1}^n\int_{B_{1/2}} |p-\overline{p}_{r_k}|^{3/2} |\nabla \varphi_k| dx + \int_{B_{1/2}} |u||f| \phi \psi_n dx. &
	\end{split}
\]
Then the bounds in \Cref{cutoff} gives
\begin{equation} \label{ckn1_claim1_eq1}
	\begin{split}
	C_3^{-1} r_{n+1}^{-2} [ A(r_{n+1}) + A_c(r_{n+1}) + D(r_{n+1}) ] \leq I_1 + I_2 + I_3 + I_4,
	\end{split}
\end{equation}
where
\[
	\begin{split}
	I_1 &= \int_{B_{1/2}} |u|^2 |\psi_n \Delta \phi + \nabla \phi \cdot \nabla \psi_n| dx, \\
	I_2 &= \sum_{k=1}^n \int_{B_{1/2}} |\nabla \varphi_k| (|u|^3dx + d\nu), \\
	I_3 &= \sum_{k=1}^n\int_{B_{1/2}} |p-\overline{p}_{r_k}|^{3/2} |\nabla \varphi_k| dx, \\
	I_4 &= \int_{B_{1/2}} |u||f| \phi \psi_n dx.
	\end{split}
\]
\par
For the term $I_1$, we use the bound in \Cref{cutoff} and H{\"o}lder inequality, i.e.
\[ I_1 \leq C_3 \int_{B_{1/2}} |u|^2 dx \leq C_3 \Big( \int_{B_{1/2}} |u|^2 dx \Big)^{2/3} \leq C_3 \epsilon^{2/3}. \]
\par
For the term $I_2+I_3$, we use the gradient bounds for $\{\varphi_k\}_{1 \leq k \leq n}$ and induction hypothesis $\{\eqref{GLclaim}_k\}_{1 \leq k \leq n}$, i.e.
\[
	\begin{split}
	I_2 + I_3 
		&\leq \sum_{k=1}^n r_k^{-5} \int_{B_{r_k}} (|u|^3dx + d\nu) + \sum_{k=1}^n r_k^{-5} \int_{B_{r_k}} |p-\overline{p}_{r_k}|^{3/2} dx \\
		&\leq \sum_{k=1}^n r_k \epsilon^{2/3}.
	\end{split}
\]
\par
For the term $I_4$, we need to decompose the integral over $B_{1/2}$ into integrals over annuli, then for each subintegral we use bounds in \Cref{cutoff}, i.e.
\[
	\begin{split}
	I_4 &= \sum_{k=2}^n \int_{ B_{r_{k-1}} \backslash B_{r_k} } |u||f| \phi \psi_n dx + \int_{B_{r_n}} |u||f| \phi \psi_n dx \\
		&\leq C_3 \sum_{k=1}^n r_k^{-4} \Big( \int_{ B_{r_k} } |u|^3 dx \Big)^{1/3} 
			\Big( \int_{ B_{r_k} } |f|^q dx \Big)^{1/q} \Big( \int_{ B_{r_k} } 1 dx \Big)^{2/3-1/q} \\
		&\leq C_3 \sum_{k=1}^n r_k^{2-6/q} \epsilon^{2/9} \kappa^{1/q}.
	\end{split}
\]
We choose $\kappa = \epsilon^{4/9}$. Because $q>3$, $2-\frac{6}{q}>0$, then
\[ I_4 \leq C_3C(q) \epsilon^{2/3}. \]
\par
Now we plug the estimates for $I_1,I_2,I_3$ and $I_4$ into \eqref{ckn1_claim1_eq1}, which concludes the proof.
\end{proof}

\begin{proof}[Proof of \Cref{ckn1_claim2}]

For simplicity, let $x_0=0$, then \Cref{Sobolev_type_inequality} yields for any $2 \leq k \leq n$,
\begin{equation} \label{ckn1_estimate_7}
	G(r_k) + G_c(r_k) \leq C_1A^{3/2}(r_k) + C_1A_c^{3/2}(r_k) + C_1D^{3/2}(r_k) \leq C_1C_B \epsilon r_k^3.
\end{equation}
\par
To bound $L(r_n)$, we use the pressure estimate \Cref{pressure_estimate_1} and let $\rho=r_1$, $r=r_n$, i.e.
\[
	\begin{split}
	L(r_n) 
	&\leq C_2 G(r_{n-1}) + C_2 r_n^{9/2} \Big( \int_{r_{n-1}<|y|<r_1} \frac{|u|^2}{|y|^7} dy \Big)^{3/2}
		+ C_2 \Big(\frac{r_n}{r_1}\Big)^{9/2} [L(r_1) + G(r_1)] \\
	&\leq C_1 C_2 C_B r_{n-1}^3 \epsilon + C_2 r_n^{9/2} \Big( \sum_{k=2}^{n-1} \int_{r_k<|y|<r_{k-1}} \frac{|u|^2}{|y|^7} dy \Big)^{3/2}
		+ C_2 \Big(\frac{r_n}{r_1}\Big)^{9/2} \epsilon \\
	&\leq 8C_2(C_1C_B+1) r_n^3 \epsilon + C_2 r_n^{9/2} \Big( \sum_{k=2}^{n-1} 16r_k^{-3} D(r_{k-1}) \Big)^{3/2} \\
	&\leq 8C_2(C_1C_B+1) r_n^3 \epsilon + C_2 r_n^{9/2} \Big( \sum_{k=2}^{n-1} C_B \epsilon^{2/3} r_k^{-1} \Big)^{3/2} \\
	&\leq 8C_2(C_1C_B+1+C_B^{3/2}) r_n^3 \epsilon.
	\end{split}
\]
In the second inequality, the first term follows from the initial smallness condition \eqref{ckn1_initial} when $n=2$, from \eqref{ckn1_estimate_7} when $n>2$, and the third term is always due the the initial smallness condition \eqref{ckn1_initial}. In the fourth inequality, we use the assumptions $\{\eqref{ADclaim}_k\}_{2 \leq k \leq n}$.
\par
Combining the estimates yields
\[ G(r_n) + G_c(r_n) + L(r_n) \leq \big[ 8C_2(C_1C_B+1+C_B^{3/2}) + C_1 C_B \big] r_n^3 \epsilon, \]
then we choose $\epsilon > 0$ such that
\[ \big[ 8C_2(C_1C_B+1+C_B^{3/2}) + C_1 C_B \big] \epsilon^{1/3} < 1. \]
This concludes the proof.

\end{proof}

\end{proof}

Next we show the second local regularity result.

\begin{proposition} \label{ckn2}
There exists an absolute constant $\tau>0$ with the following property. Suppose that we have a weak solution set $(u,p,\nu,\mu)$ of the Navier-Stokes equations in $\R^6$ satisfying
\begin{equation}
	\begin{split}
	\limsup_{r \rightarrow 0} \frac{1}{r^2} \Big[ \int_{B_r(x_0)} |\nabla u|^2 dx + \int_{B_r(x_0)} d\mu \Big] &\leq \tau
	\end{split}
\end{equation}
for a point $x_0 \in \R^6$, then $\|u\|_{L^{\infty}(B_{r_0}(x_0))} < C $ for some $r_0 > 0$ which depends on $x_0$. Moreover, the concentration measures $\mu$ and $\nu$ vanish in $B_{r_0}(x_0)$.
\end{proposition}

\begin{proof}
Without loss of generality, fix $x_0=0 \in \R^6$. For $\rho>0$, choose the cut-off function $\phi \psi_1 \in C_c^{\infty}(B_{\rho}(x_0))$ with $\rho \in (0, \frac{1}{2})$ and $\theta \in (0,\frac{1}{2}]$, then we get the following local energy inequality from \eqref{local_energy_2},
\[
	\begin{split}
	- \int_{B_{\rho}} |u|^2 \phi \Delta \psi_1 dx + 2 \int_{B_{\rho}} \phi \psi_1 (|\nabla u|^2 dx + d\mu) \leq \int_{B_{\rho}} |u|^2 |\psi_1 \Delta \phi + \nabla \phi \cdot \nabla \psi_1| dx & \\
	+ \int_{B_{\rho}} |\nabla (\phi \psi_1)| (|u-\overline{u}_{\rho}|^3dx + d\nu) + \int_{\R^6} |u|^2 (u \cdot \nabla) (\phi \psi_1) dx &\\
	+ \int_{\R^6} |p| |u| |\nabla (\phi \psi_1)| dx + \int_{B_{\rho}} |u||f| \phi \psi_1 dx &.
	\end{split}
\]
This yields
\begin{equation} \label{ckn2_eq1}
	\begin{split}
	C_3^{-1} (\theta \rho)^{-2} [ A(\theta \rho) + A_c(\theta \rho) + D(\theta \rho) ] \leq I_1 + I_2 + I_3 + I_4 + I_5,
	\end{split}
\end{equation}
where
\[
	\begin{split}
	I_1 &= \int_{B_{\rho}} |u|^2 |\psi_1 \Delta \phi + \nabla \phi \cdot \nabla \psi_1| dx, \\
	I_2 &= \int_{B_{\rho}} |\nabla (\phi \psi_1)| (|u-\overline{u}_{\rho}|^3dx + d\nu), \\
	I_3 &= \int_{\R^6} |u| |p| |\nabla (\phi \psi_1)| dx, \\
	I_4 &= \int_{B_{\rho}} |u||f| \phi \psi_1 dx, \\
	I_5 &= \int_{\R^6} |u|^2 (u \cdot \nabla) (\phi \psi_1) dx.
	\end{split}
\]
\par
For $I_1,I_3$ and $I_4$, we simply use H{\"o}lder’s inequality and the bounds in \Cref{cutoff}, i.e.
\begin{equation} \label{ckn2_eq2}
	\begin{split}
	(\theta \rho)^2 I_1 &\leq C \theta^2 G^{2/3}(\rho), \\
	(\theta \rho)^2 I_3 &\leq C \theta^{-3} G^{1/3}(\rho) K^{2/3}(\rho), \\
	(\theta \rho)^2 I_4 &\leq C \theta^{-2} G^{1/3}(\rho) F_2^{1/2}(\rho).
	\end{split}
\end{equation}
\par

For $I_2$, \Cref{Sobolev_type_inequality} yields
\begin{equation} \label{ckn2_eq3}
	(\theta \rho)^2 I_2 \leq C \theta^{-3} [ A^{3/2}(\rho) + A_c^{3/2}(\rho) ].
\end{equation}
\par

For $I_5$, we have
\begin{equation} \label{ckn2_eq4}
	\begin{split}
	I_5 =& \int_{\R^6} |u - \overline{u}_{\rho}|^2 \big[ (u-\overline{u}_{\rho}) \cdot \nabla \big] (\phi \psi_1) dx 
		+ \int_{\R^6} |u - \overline{u}_{\rho}|^2 (\overline{u}_{\rho} \cdot \nabla) (\phi \psi_1) dx \\
    	&+ 2\int_{\R^6} \big[(u-\overline{u}_{\rho}) \cdot \overline{u}_{\rho} \big] (u \cdot \nabla) (\phi \psi_1) dx \\
      \leq& I_2 + \int_{\R^6} |u - \overline{u}_{\rho}|^2 (\overline{u}_{\rho} \cdot \nabla) (\phi \psi_1) dx
    	+ 2\int_{\R^6} \big[(u-\overline{u}_{\rho}) \cdot \overline{u}_{\rho} \big] (u \cdot \nabla) (\phi \psi_1) dx \\
    =& I_2 - \int_{\R^6} (\overline{u}_{\rho} \cdot \nabla)u \cdot (u - \overline{u}_{\rho}) \phi \psi_1 dx
    	- 2\int_{\R^6} (u \cdot \nabla)u \cdot \overline{u}_{\rho} \phi \psi_1 dx \\
    \leq& I_2 + 3 \rho^{-6} \Big( \int_{B_{\rho}} |u|^3 dx \Big)^{1/3} \Big( \int_{B_{\rho}} |\nabla u|^2 dx \Big)^{1/2} \Big( \int_{B_{\rho}} |u|^2 dx \Big)^{1/2} \\
    \leq& I_2 + 3 \rho^{-2} G^{1/3}(\rho) A^{1/2}(\rho) D^{1/2}(\rho).
  \end{split}
\end{equation}
\par

Plug \eqref{ckn2_eq2}, \eqref{ckn2_eq3} and \eqref{ckn2_eq4} into \eqref{ckn2_eq1}, we get
\begin{equation} \label{ckn2_eq5}
	\begin{split}
	A(\theta \rho) + A_c(\theta \rho) + D(\theta \rho)
		\leq& C \theta^2 G^{2/3}(\rho) + C \theta^{-2} G^{1/3}(\rho) F_2^{1/2}(\rho)
				 + C \theta^{-3} G^{1/3}(\rho) K^{2/3}(\rho) \\
			&+ C \theta^{-3} [ A^{3/2}(\rho) + A_c^{3/2}(\rho) ] + C \theta^2 G^{1/3}(\rho) A^{1/2}(\rho) D^{1/2}(\rho) \\
		\leq& C \theta^2 G^{2/3}(\rho) + C \theta^{-8} K^{4/3}(\rho) + C \theta^{-6} F_2(\rho) \\
			&+ C \theta^{-3} [ A^{3/2}(\rho) + A_c^{3/2}(\rho) ] + C \theta^2 A(\rho) D(\rho)\\
		\leq& C \theta^2 A(\rho) + C \theta^2 D(\rho) + C \theta^{-8} K^{4/3}(\rho) + C \theta^{-6} F_2(\rho) \\
			&+ C \theta^{-3} [ A^{3/2}(\rho) + A_c^{3/2}(\rho) ].
	\end{split}
\end{equation}
In the second inequality, we use Young's inequality to combine the terms involving $G$. The third inequlity follows from Sobolev inequality in \Cref{Sobolev_type_inequality}. Also, we absorb $C \theta^2 A(\rho) D(\rho)$ into $C \theta^2 D(\rho)$ since $A(\rho)$ is bounded from the definition.
\par

Note that \Cref{pressure_estimate_2} yields
\begin{equation} \label{ckn2_eq4}
	\begin{split}
	\theta^{-9} K^{4/3}(\theta \rho) &\leq C_2 \theta^{-13} A^2(\rho) + C_2 \theta^{-5} K^{4/3}(\rho).
	\end{split}
\end{equation}
Taking the sum of \eqref{ckn2_eq3} and \eqref{ckn2_eq4} yields
\begin{equation} \label{ckn2_eq5}
	E(\theta \rho) \leq C \theta E(\rho) + C \theta^{-6} F_2(\rho) + C_2 \theta^{-13} A^2(\rho) + C \theta^{-3} A_c^{3/2}(\rho).
\end{equation}
where $E(x_0,\rho) := A(x_0,\rho) + A_c(x_0,\rho) + D(x_0,\rho) + \theta^{-8} K^{4/3}(\rho)$. Here we use the fact $\theta \in (0,\frac{1}{2})$ to combine some terms of lower order.
\par

Finally, we choose $\theta$ small enough such that $C\theta<\frac{1}{2}$. Since
\[ \limsup_{r \rightarrow 0} A(r)+A_c(r) \leq \tau, \quad \lim_{r \rightarrow 0} F_2(r) =0, \]
we can choose $\tau$ suitably small such that there exists $\rho_0>0$ such that for any $\rho \in (0, \rho_0)$,
\[ C \theta^{-6} F_2(\rho) + C_2 \theta^{-13} A^2(\rho) + C \theta^{-3} A_c^{3/2}(\rho) \leq \frac{\tau'}{2} \]
with $\tau'$ to be determined. Therefore, \eqref{ckn2_eq5} becomes
\[ E(\theta \rho) \leq \frac{1}{2} E(\rho) + \frac{\tau'}{2} \quad \text{for any } \rho \in (0, \rho_0). \]
We could choose $k_0$ large enough such that
\[ E(\theta^{k_0} \rho_0) \leq \frac{1}{2^{k_0}} E(\rho) + \tau' \leq 2 \tau'. \]
Note that we can bound $G+G_c$ by $A+A_c+D$ by means of \Cref{Sobolev_type_inequality}, then we can bound $G(\theta^k_0 \rho_0)+G_c(\theta^k_0 \rho_0)+K(\theta^k_0 \rho_0)$ by $E(\theta^k_0 \rho_0)$. Now we can choose $\tau'$ small enough to ensure the smallness condition \eqref{ckn1_initial}, then we apply \Cref{ckn1_scaled} to conclude the proof. Note that $3q-6>0$, so the smallness condition about $f$ will be satisfied for certain radius anyway.
\end{proof}

Finally, we estimate the size of the singular set of weak solution sets. We give a definition to the singular set of weak solution sets.

\begin{definition} \label{singularset}
Suppose that $(u,p,\nu,\mu)$ is a weak solution set of the stationary Navier-Stokes equation in $\R^6$. A point $x_0 \in \R^6$ is called a regular point if $u \in L^{\infty}(B_r(x_0))$ for some $r>0$. Otherwise, $x_0$ is called a singular point. The singular set is the set of all singular points.
\end{definition}

\begin{proof}[Proof of \Cref{mainregularity} and \Cref{maintheorem}]
\Cref{mainregularity} follows from \Cref{ckn2} a standard covering argument and the fact
\[ \int_{\R^6} |\nabla u|^2 dx + \int_{\R^6} d \mu \leq \liminf_{k \rightarrow \infty} \int_{\R^6} |\nabla u_k|^2 dx < \infty. \]
For the standard covering argument, we refer to \cite{caffarelli1982partial} or \cite{wu2021partially}.
\par

For \Cref{maintheorem}, the existence is proved in \Cref{preparation_existence}. \Cref{remark_distribution_solution} shows that $(u,p)$ satisfies the stationary Navier-Stokes equations in the sense of distributions. \Cref{measure_convergence_3} proves the local energy inequalities. The partial regularity of $u$ is contained in \Cref{mainregularity}.
\end{proof}

\renewcommand{\bibname}{References}
\bibliography{bibliography}
\bibliographystyle{abbrv}

\end{document}